\documentclass[10pt]{amsart}
\usepackage{amsmath,amssymb,latexsym,hyperref,url,graphicx}

\usepackage{cancel}

\graphicspath{{figures/}}

\newtheorem{conjecture}{Conjecture}
\newtheorem*{lemma*}{Lemma}
\newtheorem*{proposition*}{Proposition}
\newtheorem*{theorem*}{Theorem}

\newcommand{\R}{\mathbb{R}}
\newcommand{\Z}{\mathbb{Z}}
\newcommand{\Mod}{\textnormal{mod}}
\newcommand{\Frac}{\textnormal{frac}}

\begin{document}

\title{Two binomial coefficient conjectures}
\author{Eric Rowland}

\address{
	Mathematics Department \\
	Tulane University \\
	New Orleans, LA 70118, USA
}
\date{February 7, 2011}

\begin{abstract}
Much is known about binomial coefficients where primes are concerned, but considerably less is known regarding prime powers and composites.
This paper provides two conjectures in these directions, one about counting binomial coefficients modulo $16$ and one about the value of $\binom{n}{2p}$ modulo $n$.
\end{abstract}

\dedicatory{
To Doron Zeilberger in honor of his 60th birthday!
}

\maketitle

\section{Introduction}\label{Introduction}

On the occasion of Doron Zeilberger's sixtieth birthday I present two conjectures on arithmetic properties of binomial coefficients.
These conjectures epitomize the celebrated Zeilbergerian experimental approach to mathematics.
They began as naive questions.
Answers to these questions were computed (by a machine) in many specific instances, and the data was analyzed (by a human with the assistance of a machine) for patterns.
From such experiments, and in the case of the second conjecture after a fair amount of work, the author formulated and simplified the statements of the conjectures, which were then verified against additional data.

Perhaps most importantly, the conjectures given in this paper are empirical claims and are unproven!
We do however condescend to prove some propositions.

Accompanying these two conjectures are two stories.
Arithmetic properties of binomial coefficients have a long and interesting history, a major theme of which is that properties of $\binom{n}{m}$ involving the prime $p$ are related to the standard base-$p$ representations $n_l n_{l-1} \cdots n_1 n_0$ and $m_l m_{l-1} \cdots m_1 m_0$ of $n$ and $m$.
The theorems of Kummer~\cite[pages~115--116]{Kummer} and Lucas~\cite{Lucas} are important examples of this relationship.

\begin{theorem*}[Kummer]
Let $p$ be a prime, and let $0 \leq m \leq n$.
The exponent of the highest power of $p$ dividing $\binom{n}{m}$ is the number of borrows involved in subtracting $m$ from $n$ in base $p$.
\end{theorem*}

\begin{theorem*}[Lucas]
If $p$ is a prime and $0 \leq m \leq n$, then
\[
	\binom{n}{m} \equiv \prod_{i=0}^l \binom{n_i}{m_i} \mod p.
\]
\end{theorem*}

The relationship between $\binom{n}{m}$ and the base-$p$ representations of $n$ and $m$ can be thought of as a consequence of the self-similarity present in both Pascal's triangle and base-$k$ representations of nonnegative integers.
Figure~\ref{mod powers of 2} shows Pascal's triangle modulo powers of $2$.

Moreover, this relationship provides a natural rendering of number theoretic questions about binomial coefficients in terms of combinatorial questions about words.
We shall use the notation $|n|_w$ to denote the number of occurrences of the word $w$ in the base-$p$ representation $n_l n_{l-1} \cdots n_1 n_0$, where $p$ is the prime implied by context.

\begin{figure}
	\includegraphics{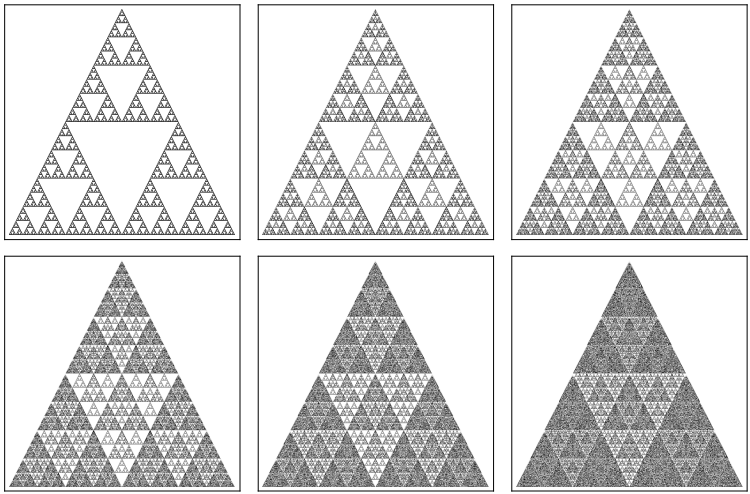}
	\caption{Pascal's triangle modulo $2^\alpha$ for $1 \leq \alpha \leq 6$, in which nested structure can be seen.}
	\label{mod powers of 2}
\end{figure}

Both conjectures in the paper are concerned with the values of $\binom{n}{m}$ modulo $k$.
The first conjecture concerns counting binomial coefficients $\binom{n}{m}$ equivalent to a certain residue class modulo $16$, for fixed $n$.
The second conjecture concerns the value of $\binom{n}{2p}$ modulo $n$.

\section{Counting entries modulo $k$}

Counting binomial coefficients is a sort of ``meta-combinatorics'' --- we are counting the counters!
One basic question we can ask about Pascal's triangle modulo $k$ is how many nonzero entries there are on row $n$.
This question was answered for prime $k = p$ by Fine~\cite{Fine} and shown to be $\prod_{i=0}^l \left(n_i + 1\right)$.
For prime powers $k = p^\alpha$ an expression was found by the author~\cite{Rowland}.
There has been no progress in the case of composite $k$.

Here we shall be interested in a refinement of the previous question:  Let $a_{k,r}(n)$ be the number of integers $0 \leq m \leq n$ such that $\binom{n}{m} \equiv r \mod k$; what is the structure of $a_{k,r}(n)$?

The first result on $a_{k,r}(n)$ was Glaisher's 1899~\cite{Glaisher} discovery that $a_{2,1}(n) = 2^{|n|_1}$.
After that there appears to have been no results until 1978, when Hexel and Sachs~\cite[\S5]{Hexel--Sachs} determined a formula for $a_{p,r^i}(n)$ for any prime $p$ in terms of $(p-1)$th roots of unity, where $r$ is a primitive root modulo $p$.
From this they derived
\begin{align*}
	a_{3,1}(n) &= 2^{|n|_1 - 1} \cdot \left(3^{|n|_2} + 1\right) \\
	a_{3,2}(n) &= 2^{|n|_1 - 1} \cdot \left(3^{|n|_2} - 1\right)
\end{align*}
as well as explicit formulas for $a_{5,r^i}(n)$ in terms of $|n|_1$, $|n|_2$, $|n|_3$, and $|n|_4$.
Garfield and Wilf~\cite{Garfield--Wilf} exhibited a method to compute the generating function $\sum_{i=0}^{p-2} a_{p,r^i}(n) x^i$, where again $r$ is a primitive root.

Recently, Amdeberhan and Stanley~\cite[Theorem~2.1]{Amdeberhan--Stanley} studied the number of coefficients equal to $r$ in the $n$th power of a general multivariate polynomial over a finite field, where $r$ is an invertible element of the field.
They showed that if $f(\textbf{x}) \in \mathbb{F}_q[x_1, \dots, x_m]$ then the number $a(n)$ of coefficients in $f(\textbf{x})^n$ equal to $r \in \mathbb{F}_q^\times$ is a $q$-regular sequence, meaning that it satisfies linear recurrences in $a(q^e n + i)$~\cite{Allouche--Shallit}.

The first results for a non-prime modulus were formulas for $a_{4,r}(n)$ given by Davis and Webb~\cite{Davis--Webb 1989}:

\begin{align*}
	a_{4,1}(n) &=
		\begin{cases}
			2^{|n|_1}	& \text{if $|n|_{11} = 0$} \\
			2^{|n|_1 - 1}	& \text{otherwise}
		\end{cases} \\
	a_{4,2}(n) &= 2^{|n|_1 - 1} \cdot |n|_{10} \\
	a_{4,3}(n) &=
		\begin{cases}
			0	& \text{if $|n|_{11} = 0$} \\
			2^{|n|_1 - 1}	& \text{otherwise}
		\end{cases}
\end{align*}
(where $|n|_w$ counts occurrences of $w$ in the base-$2$ representation of $n$).
In particular, if $r$ is odd then $a_{4,r}(n)$ is either $0$ or a power of $2$.

Granville observed that in a sense the statement that $a_{4,r}(n)$ is either $0$ or a power of $2$ parallels Glaisher's result that $a_{2,1}(n) = 2^{|n|_1}$ is (either $0$ or) a power of $2$.
Granville asked whether a similar statement is true modulo $8$.
His delightful paper ``Zaphod Beeblebrox's brain and the fifty-ninth row of Pascal's triangle''~\cite{Granville 1992} contains the following theorem.

\begin{theorem*}[Granville]
If $r$ is odd then $a_{8,r}(n)$ is either $0$ or a power of $2$.
\end{theorem*}

As in the case of $a_{4,r}(n)$, for even $r$ other values occur.
For example, $a_{8,4}(12) = 5$ and $a_{8,0}(16) = 12$.

Granville goes on to consider the number of binomial coefficients congruent to an odd residue class $r$ modulo $16$, and he discovers that the analogous statement fails for $n = 59$!
Granville writes: ``Unbelievably, there are exactly six entries of Row $59$ in each of the congruence classes $1$, $11$, $13$, and $15 \pmod {16}$! Our pattern has come to an end, but not before providing us with some interesting mathematics, as well as a couple of pleasant surprises.''

But this isn't the end of the story!
We are now lead to ask:  What values \emph{does} $a_{16,r}(n)$ take for odd $r$?
Explicit computation shows that up to row $2^{20}$ the set of values that occur is
\[
	\{0, 1, 2, 4, 6, 8, 12, 16, 20, 24, 32, 40, 48, 56, 64, 72, 80, 96, \dots, 65536\}.
\]
This computation took nearly three CPU weeks to complete.
It is clear that the only odd number that occurs is $1$, since $\binom{2n}{n}$ is even for $n \geq 1$, so $a_{16,r}(n)$ is even.
But the data suggests that $a_{16,r}(n)$ does not take on every even number.

\begin{conjecture}
Fix $n$ and odd $r$.

If $a_{16,r}(n)$ is divisible by $\phantom{0}3$, then it is also divisible by $2$.

If $a_{16,r}(n)$ is divisible by $\phantom{0}5$, then it is also divisible by $2^2$.

If $a_{16,r}(n)$ is divisible by $\phantom{0}7$, then it is also divisible by $2^3$.

If $a_{16,r}(n)$ is divisible by $11$, then it is also divisible by $2^5$.

If $a_{16,r}(n)$ is divisible by $13$, then it is also divisible by $2^6$.

If $a_{16,r}(n)$ is divisible by $17$, then it is also divisible by $2^4$.

If $a_{16,r}(n)$ is divisible by $31$, then it is also divisible by $2^5$.
\end{conjecture}

Moreover, these eight primes are the only primes that appear up to row $2^{20}$.
Of course, there is a probably a more general and more compact conjecture waiting to be discovered.
However, it is not clear what the corresponding powers of $2$ might be for larger primes.

Granville's analysis of $a_{8,r}(n)$ was long (and in fact incomplete~\cite{Granville 1997 correction, Huard--Spearman--Williams mod 8}).
However, it should be possible to automate Granville's work in true Zeilberger fashion.
We do not undertake this work here, but once this is done and understood, the methodology will likely be applicable to $a_{16,r}(n)$, and perhaps the statements in the conjecture and the larger context surrounding them will become clear.
Hopefully a reader will take up this challenge!

\section{The value of $\binom{n}{m}$ modulo $n$}

We all know the feeling of anxiety caused by witnessing a calculus student apply the ``high school dream''\footnote{known as the ``freshman dream'' in countries such as the U.S. with lagging math education}
\[
	(a + b)^n = a^n + b^n,
\]
in which exponentiation distributes over addition.
But of course the high school dream is actually valid on occasion.
In particular, $(a + b)^p \equiv a^p + b^p \mod p$, since $\binom{p}{m} \equiv 0 \mod p$ for $1 \leq m \leq p-1$.

We might then ask:
How badly does the high school dream fail for non-primes?
What is $\binom{n}{m}$ modulo $n$, and how frequently is it nonzero?
With such an answer we can precisely quantify how upset at students we are entitled to get!

Rather than study $\binom{n}{m}$ modulo $n$ directly, we divide by $n$ and consider $\Frac(\frac{1}{n} \binom{n}{m})$, where $\Frac(x) = x - \lfloor x \rfloor$ is the fractional part of $x$.
We can recover the value of $\binom{n}{m}$ modulo $n$ from $\binom{n}{m} \equiv n \, \Frac(\frac{1}{n} \binom{n}{m}) \mod n$.
Figure~\ref{n choose m mod n} displays $\Frac(\frac{1}{n} \binom{n}{m})$ graphically.
The principal advantage of reformulating in this way is that for fixed $m \geq 1$ the sequence $\Frac(\frac{1}{n} \binom{n}{m})$ for $n \geq 1$ is periodic, which we soon prove.
For example, if $m = 4$ the sequence $\Frac(\frac{1}{n} \binom{n}{m})$ is
\[
	0, 0, 0, \frac{1}{4}, 0, \frac{1}{2}, 0, \frac{3}{4}, 0, 0, 0, \frac{1}{4}, 0, \frac{1}{2}, 0, \frac{3}{4}, 0, 0, 0, \frac{1}{4}, 0, \frac{1}{2}, 0, \frac{3}{4}, \dots.
\]
If $m = 5$ the sequence is
\[
	0, 0, 0, 0, \frac{1}{5}, 0, 0, 0, 0, \frac{1}{5}, 0, 0, 0, 0, \frac{1}{5}, 0, 0, 0, 0, \frac{1}{5}, 0, 0, 0, 0, \frac{1}{5}, \dots,
\]
and if $m = 6$ it is
\begin{align*}
	&0, 0, 0, 0, 0, \frac{1}{6}, 0, \frac{1}{2}, \frac{1}{3}, 0, 0, 0, 0, \frac{1}{2}, \frac{2}{3}, \frac{1}{2}, 0, \frac{1}{3}, 0, 0, 0, \frac{1}{2}, 0, \frac{1}{6}, \\
	&0, 0, \frac{1}{3}, 0, 0, \frac{1}{2}, 0, \frac{1}{2}, \frac{2}{3}, 0, 0, \frac{1}{3}, 0, \frac{1}{2}, 0, \frac{1}{2}, 0, \frac{2}{3}, 0, 0, \frac{1}{3}, \frac{1}{2}, 0, \frac{1}{2}, \\
	&0, 0, \frac{2}{3}, 0, 0, \frac{5}{6}, 0, \frac{1}{2}, 0, 0, 0, \frac{2}{3}, 0, \frac{1}{2}, \frac{1}{3}, \frac{1}{2}, 0, 0, 0, 0, \frac{2}{3}, \frac{1}{2}, 0, \frac{5}{6}, \dots
\end{align*}
with period length $72$.

The proof of periodicity follows from a result of Lu and Tsai~\cite{Lu--Tsai}.

\begin{figure}
	\includegraphics{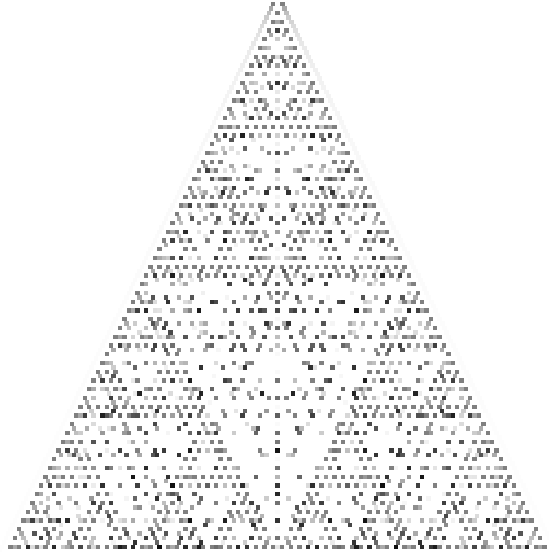}
	\caption{Rows $1$ through $2^7 - 1$ of Pascal's triangle, where $\binom{n}{m}$ is assigned the gray value between $0$ (white) and $1$ (black) corresponding to $\Frac(\frac{1}{n} \binom{n}{m})$.}
	\label{n choose m mod n}
\end{figure}

\begin{theorem*}[Lu and Tsai]
Let $k \geq 1$ and $m \geq 1$.
The sequence $\binom{n}{k}$ modulo $m$ for $n \geq 0$ is periodic with (minimal) period length
\[
	m \cdot \prod_{p \mid m} p^{\lfloor \log_p k \rfloor}.
\]
\end{theorem*}

We now apply the Lu--Tsai result to $\binom{n}{m} \mod n$.

\begin{proposition*}
Let $m \geq 1$.
The sequence $\Frac(\frac{1}{n} \binom{n}{m})$ for $n \geq 1$ is periodic with period length
\[
	m \cdot \prod_{p \mid m} p^{\lfloor \log_p (m-1) \rfloor}.
\]
\end{proposition*}

The factor $p^{\lfloor \log_p (m-1) \rfloor}$ can be interpreted as the largest power of $p$ that is strictly less than $m$.

\begin{proof}
We have
\[
	\frac{1}{n} \binom{n}{m} = \frac{1}{m} \binom{n-1}{m-1}.
\]
Taking the fractional part of both sides and letting $k = m-1$ in the previous theorem gives the desired result.
\end{proof}

In the remainder of the paper we aim to compute $\Frac(\frac{1}{n} \binom{n}{m})$ in two cases --- when $m = p$ and when $m = 2 p$.
The case when $m = p$ is fairly simple.
Let $\delta_S$ be $1$ if the statement $S$ is true and $0$ if $S$ is false.

\begin{proposition*}
If $p$ is a prime and $n \geq 1$, then
\[
	\Frac(\frac{1}{n} \binom{n}{p}) =
	\frac{1}{p} \delta_{p \mid n} =
	\begin{cases}
		0	& \text{if $p \nmid n$} \\
		1/p	& \text{if $p \mid n$.}
	\end{cases}
\]
\end{proposition*}

\begin{proof}[First proof]
By the previous proposition, the period length of $\Frac(\frac{1}{n} \binom{n}{p})$ is $p$, so the values of $\Frac(\frac{1}{n} \binom{n}{p})$ for $n \geq 1$ are determined by the values for $1 \leq n \leq p$.
For $1 \leq n \leq p-1$ we have $\binom{n}{p} = 0$, and $\binom{p}{p} = 1$.
\end{proof}

Therefore $\binom{n}{p}$ is zero modulo $n$ with frequency $\frac{p-1}{p}$.

We provide a second proof of this proposition that is more similar to the proof of the final proposition of the section.
We require two preliminary results.
The first follows immediately from Kummer's theorem.
Let $\nu_p(n)$ be the exponent of the highest power of $p$ dividing $n$.

\begin{lemma*}
If $p$ is a prime and $0 \leq m \leq n$, then $\nu_p(\binom{p n}{p m}) = \nu_p(\binom{n}{m})$.
\end{lemma*}

The second is a result discovered by Anton in 1869 and rediscovered by both Stickelberger and Hensel that determines the value modulo $p$ of a binomial coefficient removed of its $p$ factors~\cite{Dickson I, Granville 1997}.
Recall that $n_l n_{l-1} \cdots n_1 n_0$ and $m_l m_{l-1} \cdots m_1 m_0$ are the representations of $n$ and $m$ in base $p$.
We write $(n-m)_i$ for the $i$th base-$p$ digit of $n-m$.

\begin{theorem*}[Anton]
Let $p$ be a prime, and let $0 \leq m \leq n$.
We have
\[
	\frac{1}{(-p)^{\nu_p(\binom{n}{m})}} \binom{n}{m} \equiv \prod_{i=0}^l \frac{n_i!}{m_i! (n-m)_i!} \mod p.
\]
\end{theorem*}

Note that the arithmetic on the right side of the congruence takes place in $\Z/p\Z$, not $\R/p\Z$.

We are now ready to give the second proof.

\begin{proof}[Second proof of the proposition]
We show that $\binom{n}{p} \equiv \frac{n}{p} \delta_{p \mid n} \mod n$.

For any integer $m$, $\frac{m}{n} \binom{n}{m} = \binom{n-1}{m-1} \in \Z$, so in particular $p \binom{n}{p} \equiv 0 \mod n$.

If $p \nmid n$ then $p$ is invertible modulo $n$, so $\binom{n}{p} \equiv 0 \mod n$.

If $p \mid n$, let $\nu = \nu_p(n)$ and write $n = p^\nu n'$.
We show that $\binom{n}{p} \equiv \frac{n}{p} \mod n$.
Since $n' \mid \binom{n}{p}$ and $n' \mid \frac{n}{p}$, we have $\binom{n}{p} \equiv \frac{n}{p} \mod n'$.
It remains to show that $\binom{n}{p} \equiv \frac{n}{p} \mod p^\nu$.
By the lemma, $\nu_p(\binom{n}{p}) = \nu_p(\binom{n/p}{1}) = \nu_p(n) - 1$.
Write $\binom{n}{p} = p^{\nu - 1} b'$.
One checks that the ratio in Anton's theorem is
\[
	\frac{n_i!}{p_i! (n-p)_i!}
	\equiv
	\begin{cases}
		1	& \text{if $i = 0$} \\
		-1	& \text{if $1 \leq i < \nu$} \\
		n_\nu	& \text{if $i = \nu$} \\
		1	& \text{if $i > \nu$}
	\end{cases}
	\mod p.
\]
It follows that $b' \equiv n' \mod p$, so multiplying by $p^{\nu - 1}$ gives $\binom{n}{p} \equiv \frac{n}{p} \mod p^\nu$.
\end{proof}

When $m$ is not prime, the behavior of $\Frac(\frac{1}{n} \binom{n}{m})$ is more complex.
Our second conjecture concerns the case $m = 2p$.
The period length of $\Frac(\frac{1}{n} \binom{n}{2p})$ for $p \neq 2$ is $2^{\lfloor \log_2 (2p-1) \rfloor + 1} p^2$.
However, the conjecture states that we may specify these $2^{\lfloor \log_2 (2p-1) \rfloor + 1} p^2$ values in a more concise way than by listing them explicitly.
First we prove a result for $\Frac(\frac{2}{n} \binom{n}{2p})$.

\begin{proposition*}
If $p$ is a prime and $n \geq 1$, then
\[
	\Frac(\frac{2}{n} \binom{n}{2 p}) =
	\Frac(\frac{n_1 - 1}{p}) \delta_{p \mid n}.
\]
\end{proposition*}

For example, if $p = 3$ then the sequence $\Frac(\frac{2}{n} \binom{n}{2p})$ is
\[
	0, 0, 0, 0, 0, \frac{1}{3}, 0, 0, \frac{2}{3}, 0, 0, 0, 0, 0, \frac{1}{3}, 0, 0, \frac{2}{3}, 0, 0, 0, 0, 0, \frac{1}{3}, 0, 0, \frac{2}{3}, \dots.
\]

\begin{proof}
As in the proof of the previous proposition, $\frac{2 p}{n} \binom{n}{2 p} \in \Z$, so $2 p \binom{n}{2 p} \equiv 0 \mod n$.

First consider $p = 2$.
The period length of $\Frac(\frac{2}{n} \binom{n}{4})$ is $4$.
The period length of $\Frac(\frac{n_1 - 1}{2}) \delta_{2 \mid n}$ is also $4$, so one checks that the first $4$ values of the two sequences agree.

Let $p \neq 2$.

If $p \nmid n$ then $2 \binom{n}{2 p} \equiv 0 \mod n$.

If $p \mid n$, let $\nu = \nu_p(n)$ and write $n = p^\nu n'$.
We show that $2 \binom{n}{2 p} \equiv \frac{n}{p} (\frac{n}{p} - 1) \mod n$.
Since $2 \binom{n}{2 p} \equiv 0 \mod n'$ and $\frac{n}{p} (\frac{n}{p} - 1) \equiv 0 \mod n'$, it suffices to show that $2 \binom{n}{2 p} \equiv \frac{n}{p} (\frac{n}{p} - 1) \mod p^\nu$.
Let $\beta = \nu_p(2 \binom{n}{2 p})$, and write $2 \binom{n}{2 p} = p^\beta b'$.
By the lemma, $\beta = \nu_p(\frac{n}{p} (\frac{n}{p} - 1)) = \nu - 1 + \nu_p(\frac{n}{p} - 1)$.
We consider two cases.
If $n_1 \geq 2$, the ratio in Anton's theorem is
\[
	\frac{n_i!}{(2p)_i! (n-2p)_i!}
	\equiv
	\begin{cases}
		1						& \text{if $i = 0$} \\
		\frac{1}{2} \frac{n}{p} (\frac{n}{p} - 1)	& \text{if $i = 1$} \\
		1						& \text{if $i > 1$}
	\end{cases}
	\mod p;
\]
in this case $\nu = 1$ and $\beta = 0$, so
\[
	2 \binom{n}{2p} \equiv \frac{n}{p} \left(\frac{n}{p} - 1\right) \mod p
\]
as desired.
Alternatively, if $n_1 < 2$, the ratio is
\[
	\frac{n_i!}{(2p)_i! (n-2p)_i!}
	\equiv
	\begin{cases}
		1				& \text{if $i = 0$} \\
		\frac{1}{2! (n_1 + p - 2)!}	& \text{if $i = 1$} \\
		-1				& \text{if $2 \leq i \leq \beta$} \\
		\lfloor \frac{n}{p^i} \rfloor	& \text{if $i = \beta + 1$} \\
		1				& \text{if $i > \beta + 1$}
	\end{cases}
	\mod p,
\]
so (since $(n_1 + p - 2)! \equiv (-1)^{n_1} \mod p$)
\[
	2 \binom{n}{2p} \equiv -(-1)^{n_1} p^\beta \left\lfloor \frac{n}{p^{\beta + 1}} \right\rfloor \mod p^{\beta + 1}.
\]
Now if $n_1 = 1$ then $\beta = \nu_p(n-p) - 1$, so this becomes
\[
	2 \binom{n}{2p}
	\equiv p^{\nu_p(n-p) - 1} \left\lfloor \frac{n}{p^{\nu_p(n-p)}} \right\rfloor
	= \frac{n-p}{p}
	\equiv \frac{n}{p} \left(\frac{n}{p} - 1\right)
	\mod p^{\nu_p(n-p)},
\]
which is sufficient since $\nu_p(n-p) > 1 = \nu_p(n) $.
If on the other hand $n_1 = 0$ then $\beta = \nu_p(n) - 1$, so we have
\[
	2 \binom{n}{2p}
	\equiv -p^{\nu_p(n) - 1} \left\lfloor \frac{n}{p^{\nu_p(n)}} \right\rfloor
	= -\frac{n}{p}
	\equiv \frac{n}{p} \left(\frac{n}{p} - 1\right)
	\mod p^{\nu_p(n)}.
	\qedhere
\]
\end{proof}

Since $\Frac(\frac{2}{n} \binom{n}{2 p}) = \Frac(\frac{n_1 - 1}{p}) \delta_{p \mid n}$, it follows that, for a given $n$ and $p$, $\Frac(\frac{1}{n} \binom{n}{2 p})$ is either $\Frac(\frac{n_1 - 1}{2 p} \delta_{p \mid n})$ or $\Frac(\frac{n_1 - 1}{2 p} \delta_{p \mid n} + \frac{1}{2})$, depending on whether $\frac{2}{n} \binom{n}{2 p} - \frac{n_1 - 1}{p} \delta_{p \mid n}$ is even or odd.
For example, if $p = 3$ then $\frac{2}{n} \binom{n}{2 p} - \frac{n_1 - 1}{p} \delta_{p \mid n}$ is even unless $n$ is congruent to one of
\[
	8, 9, 14, 15, 16, 18, 22, 27, 30, 32, 33, 36, 38, 40, 42, 45, 46, 48, 51, 56, 60, 62, 63, 64, 69, 70
\]
modulo $72$.

However, we can significantly simplify this condition by rewriting the previous proposition instead as $\Frac(\frac{2}{n} \binom{n}{2 p}) = \Frac(\frac{(p + 1) (n_1 - 1)}{p}) \delta_{p \mid n}$.
Now $\Frac(\frac{1}{n} \binom{n}{2 p})$ is either $\Frac(\frac{(p + 1) (n_1 - 1)}{p} \delta_{p \mid n})$ or $\Frac(\frac{(p + 1) (n_1 - 1)}{p} \delta_{p \mid n} + \frac{1}{2})$, depending on whether $\frac{2}{n} \binom{n}{2 p} - \frac{(p + 1) (n_1 - 1)}{p} \delta_{p \mid n}$ is even or odd.
Taking $p = 3$ again, $\frac{2}{n} \binom{n}{2 p} - \frac{(p + 1) (n_1 - 1)}{p} \delta_{p \mid n}$ is even unless $n$ is congruent to one of $6, 8$ modulo $8$.

The following conjecture claims that we can determine when this second criterion holds --- and thereby compute $\Frac(\frac{1}{n} \binom{n}{2p})$ --- by examining the binary representation of $p$.
Let $\Mod(a,b)$ be the number $x \equiv a \mod b$ such that $0 \leq x < b$.
By `$|m|_0$' we mean the number of zeros in the standard binary representation of $m$ (with no leading zeros).
Note however that `$n_1$' still refers to the base-$p$ representation of $n$.

\begin{conjecture}
Let $p \neq 2$ be a prime, let $m = 2p$, and let $m_l m_{l-1} \cdots m_1 m_0$ be the standard base-$2$ representation of $m$.
Let $n \geq 1$.
Then
\[
	\frac{2}{n} \binom{n}{2 p} - \frac{(p + 1) (n_1 - 1)}{p} \delta_{p \mid n}
\]
is odd if and only if
\[
	n \equiv 2 p + \sum_{i=0}^l \left\lfloor \frac{j}{2^{|\Mod(2p,2^{i+1})|_0 - \delta_{i=1}}} \right\rfloor m_i 2^i	\mod 2^{\lfloor \log_2 (2p) \rfloor + 1}
\]
for some $0 \leq j \leq 2^{|2p|_0} - 1$.
\end{conjecture}

Letting $\textnormal{exceptional}_{2p}(n)$ be the statement that $\frac{2}{n} \binom{n}{2 p} - \frac{(p + 1) (n_1 - 1)}{p} \delta_{p \mid n}$ is odd, it follows from the conjecture that
\[
	\Frac(\frac{1}{n} \binom{n}{2p}) = \Frac(\frac{(p + 1) (n_1 - 1)}{2 p} \, \delta_{p \mid n} + \frac{1}{2} \, \delta_{\textnormal{exceptional}_{2p}(n)}).
\]

Let us compute some examples of the exceptional residue classes given by the conjecture.
Let $p = 3$.
For $j = 0$ we obtain
\begin{align*}
	n
	& \equiv 6 + \sum_{i=0}^2 \left\lfloor \frac{0}{2^{|\Mod(6,2^{i+1})|_0 - \delta_{i=1}}} \right\rfloor m_i 2^i	\mod 8 \\
	& = 6 + 0 \cdot 0 \cdot 2^0 + 0 \cdot 1 \cdot 2^1 + 0 \cdot 1 \cdot 2^2 \\
	& = 6.
\end{align*}
For $j = 1$ the exceptional residue class is
\begin{align*}
	n
	& \equiv 6 + \sum_{i=0}^2 \left\lfloor \frac{1}{2^{|\Mod(6,2^{i+1})|_0 - \delta_{i=1}}} \right\rfloor m_i 2^i	\mod 8 \\
	& = 6 + \left\lfloor \frac{1}{2^{|0|_0}} \right\rfloor \cdot 0 \cdot 2^0 + \left\lfloor \frac{1}{2^{|2|_0 - 1}} \right\rfloor \cdot 1 \cdot 2^1 + \left\lfloor \frac{1}{2^{|6|_0}} \right\rfloor \cdot 1 \cdot 2^2 \\
	& = 6 + \lfloor 1 \rfloor \cdot 0 \cdot 2^0 + \lfloor 1 \rfloor \cdot 1 \cdot 2^1 + \left\lfloor \frac{1}{2} \right\rfloor \cdot 1 \cdot 2^2 \\
	& = 8.
\end{align*}


In general, one can think of the sum in the conjecture as performing an operation similar to the product $\sum_{i=0}^l j m_i 2^i$ of $2p$ and $j$ in base $2$.
For example, let $p = 173$ and $j = 13$.  The product $2 p \cdot j$ is computed as follows using the standard multiplication algorithm.
\begin{align*}
	101011010_2 \\
	\underline{\times \hspace{1.8cm} 1101_2} \\
	1101\phantom{0_2} \\
	1101\phantom{000_2} \\
	1101\phantom{0000_2} \\
	1101\phantom{000000_2} \\
	\underline{+ \;\, 1101\phantom{00000000_2}} \\
	1000110010010_2
\end{align*}
To compute the expression
\[
	\sum_{i=0}^l \left\lfloor \frac{j}{2^{|\Mod(2p,2^{i+1})|_0 - \delta_{i=1}}} \right\rfloor m_i 2^i
\]
we truncate and shift each instance of $j$ as follows.
\begin{align*}
	101011010_2 \\
	\underline{\divideontimes \hspace{1.8cm} 1101_2} \\
	1101\phantom{0_2} \\
	11\cancel{01}\phantom{0_2} \\
	11\cancel{01}\phantom{00_2} \\
	1\cancel{101}\phantom{000_2} \\
	\underline{+ \;\, \cancel{1101}\phantom{0000_2}} \\
	10100010_2
\end{align*}

As with Conjecture~1, a proof of Conjecture~2 will undoubtedly provide more context for the result and will perhaps reveal a more natural formulation.
We have also not addressed the computational feasibility of computing $\Frac(\frac{1}{n} \binom{n}{2p})$ using Conjecture~2.
In particular, there is no obvious way of quickly determining whether there is an appropriate $j$ establishing that $n$ belongs to an exceptional residue class.

However, Conjecture~2 is suggestive in its departure from traditional results for prime power moduli.
The immediate question is whether it can be generalized to $\Frac(\frac{1}{n} \binom{n}{p q})$.
The author does not know the answer.

\section*{Acknowledgement}

I would like to thank Andrew Granville for several discussions about the arithmetic of binomial coefficients.

\end{document}